\documentclass[10pt]{article}

\usepackage{times}
\usepackage{amsthm}
\usepackage{amsmath}
\usepackage{amssymb}
\usepackage{mathrsfs}
\usepackage{pb-diagram}
\usepackage{color}

\usepackage{hyperref}

\renewcommand\eqref[1]{(\ref{#1})} 

\def\A{{\mathcal A}}
\def\o{\omega}
\def\O{\Omega}
\def\th{\theta}

\def\Z{\mathbb{Z}}

\def\B{{\mathbb B}}
\def\D{{\mathcal D}}

\def\H{\mathcal H}

\def\R{\mathbb{R}}

\def\CC{{\mathfrak C}}

\def\e{{\sf e}}

\def\m{{\sf m}}
\def\wm{\widehat{\sf m}}

\def\({\left(}
\def\[{\left[}
\def\){\right)}
\def\]{\right]}

\def\Si{\Sigma}

\def\sp{\mathop{\mathrm{sp}}\nolimits}

\def\G{{\sf G}}
\def\wG{\widehat{\sf{G}}}
\def\p{\parallel}
\def\<{\langle}
\def\>{\rangle}

\providecommand{\CC}{\mathfrak{C}}

\def\fscr{\mathscr}

\newtheorem{Theorem}{Theorem}[section]
\newtheorem{Remark}[Theorem]{Remark}
\newtheorem{Lemma}[Theorem]{Lemma}

\newtheorem{Corollary}[Theorem]{Corollary}
\newtheorem{Proposition}[Theorem]{Proposition}
\newtheorem{Definition}[Theorem]{Definition}
\newtheorem{Example}[Theorem]{Example}

\numberwithin{equation}{section}

\begin{document}


\title{Essential Spectrum and Fredholm Properties\\ for Operators on Locally Compact Groups}

\date{\today}

\author{M. M\u antoiu 
}
\date{\small}
\maketitle \vspace{-1cm}


\begin{abstract}
We study the essential spectrum and Fredholm properties of integral and pseudo-diferential operators associated to (maybe non-commutative) locally compact groups $\G$\,. The techniques involve crossed product $C^*$-algebras. We extend previous results on the structure of the essential spectrum to operators belonging (or affiliated) to the Schr\"odinger representation of certain crossed products. When the group $\G$ is unimodular and type I, we cover a new class \cite{MR} of pseudo-differential differential operators with operator-valued symbols involving the unitary dual of $\G$\,.
We use recent results \cite{NP,Ro} on the role of families of representations in spectral theory and the notion of quasi-regular dynamical system \cite{Wi}.
\end{abstract}


\section{Introduction}\label{duci}

The present article is dedicated to some new results on the essential spectrum and the Fredholm property of certain (mainly bounded) integral operators on locally compact groups  $\G$\,. The integral kernel of these operators have an integrable dependence of one of the variables, while the behavior in the second one is governed by a suitable $C^*$-algebra $\A$ of bounded left uniformly continuous functions on $\G$\,. In Section \ref{firica} we describe these "coefficient algebras" while in Section \ref{flekarin} we recall how they serve to define crossed product $C^*$-algebras that will be the main objects to study. The spectral properties will depend on the choice of $\A$ through the quasi-orbit structure of its Gelfand spectrum.

We insist on the fact that our groups $\G$ may not be discrete, Abelian or Lie groups. Compact groups are out of question, since they would basically lead to a trivial situation. The integral operators we study can be seen as a generalization of convolution dominated operators \cite{CWL,Lin,RRS,Wtt}, which may be essentially recovered for $\G=\mathbb Z^n$ and for the special form $\A=\ell^\infty(\Z^n)$ of the algebra (analogue operators on discrete groups are also called convolution dominated by certain authors). in Section \ref{furitar}, following the recent article \cite{MR},  we give a pseudo-differential form to these integral operators for the rather large subclass of unimodular, second countable type I groups (still containing many non-commutative examples).

We recall that the spectrum of an operator $T$ in a Hilbert space $\H$ admits the partition ${\rm sp}(T)={\rm sp}_{\rm dis}(T)\sqcup{\rm sp}_{\rm ess}(T)$\,. The point $\lambda$ belongs to {\it the discrete spectrum} ${\rm sp}_{\rm dis}(T)$ of $T$ if it is an eigenvalue of finite multiplicity and isolated from the rest of the spectrum. The other spectral points belong to the essential spectrum ${\rm sp}_{\rm esss}(T)$\,. It is known that ${\rm sp}_{\rm esss}(T)$ coincides with the usual spectrum of the image of $T$ in the Calkin algebra $\mathbb B(\H)/\mathbb K(\H)$\,, where $\mathbb B(\H)$ is the $C^*$-algebra of all the bounded linear in $\H$ and $\mathbb K(\H)$ is the ideal of all the compact operators.
The bounded operator $T$ in the Hilbert space $\H$ is {\it Fredholm} if its range is closed and its kernel and cokernel are finite-dimensional. By Atkinson's Theorem, this happens exactly when $T$ is invertible modulo compact operators, i.e. when its image in $\mathbb B(\H)/\mathbb K(\H)$ is invertible. 

So both problems have a commun setting. The Calkin algebra does not offer effective tools to treat them, but replacing $\mathbb B(\H)$ by suitable $C^*$-subalgebras can lead to an improvement. In the present work they are connected to crossed products $\A\!\rtimes\!\G$ associated to an action of our group on the coefficient $C^*$-algebra $\A$\,.

The results for an operator $H$ are written in our case in terms of a family $\{H_i\mid i\in I\,\}$ of "asymptotic operators", each one defined by a precise algorithm from a quasi-orbit of a dynamical system attached to $\A$ and, equivalently, from a certain ideal of the crossed product. Applying older techniques would lead at writing the essential spectrum of $H$ as the closure of the union of the spectra of all the operators $H_i$\,. Similarly, $H$ is Fredholm if and only if each operator $H_i$ is invertible and the norms of the inverses are uniformly bounded.
Recently a systematic investigation \cite{Ro,NP} put into clear terms conditions under which the closure and the uniform bound are not needed. We will use this, in particular conditions from \cite[Sect. 3]{NP} involving the primitive ideal space, to state our main result in Section \ref{flokoshin} and to prove it in \ref{furitut}.

The relevant information for the primitive ideal space of the crossed product relies on the notion of {\it quasi-regular dynamical system} \cite{Wi}. There is no counterexample to such a property, but for the moment, unfortunately, we have to impose quasi-regularity as in implicit condition. Some comments on sufficient conditions for quasi-regularity are contained in Section \ref{anulatus}; in particular the property holds if $\A$ is separable. Non-separable examples are given in Section \ref{anulitus}.

We partially neglect unbounded operators, although in a non-explicit way they are covered if they happen to be affiliated, i.e. if their resolvent families have the form already indicated. It is not easy to turn this into an explicit condition if the group $\G$ is too complicated. For simpler groups (Abelian for instance) affiliation is easier but many results already exist. So we will hopefully treat this issue in a future publication.

This is not the right place to try to draw a detailed history of the subject. Some basic articles and books are \cite{CWL,Da,LR,Lin,Ra,RR,RRR,RRS} (see also references therein). 
The important role that can be conferred to crossed product $C^*$-algebras has been outlined in \cite{GI1,GI2,GN,Ma}, especially in the case of Abelian groups. Twisted crossed products, needed to cover pseudo-differential operators with variable magnetic fields, are advocated in \cite{LMR,MPR2}.  Groupoid or related $C^*$-algebras have been used in \cite{LMN,LN,Ni,NP}. Te paper \cite{Ge} treats operators on suitable metric spaces. In \cite{Ma1,Ma2} we examined pseudo-differential operators in $\mathbb R^n$ with symbols $f(x,\xi)$ having a complicated behavior at infinity in "phase space", i.e. in both variables $x$ and $\xi$\,. 
The essential spectrum of operators (including Schr\"odinger Hamiltonians) has also been treated with geometrical methods (without using $C^*$-algebras at all) in \cite{HM,LS}.

By rather similar techniques one can prove localization (non-propagation) results, adapting to the non-commutative setting the approach of \cite{AMP,MPR2}.

\medskip
{\bf Acknowledgements:} The author has been supported by N\'ucleo Milenio de F\'isica Matem\'atica RC120002 and the Fondecyt Project 1120300.

\section{Left-invariant $C^*$-algebras of functions on a locally compact group}\label{firica}

Let $\G$ be a locally compact group with unit $\e$\,, modular function $\Delta$ and fixed left Haar measure $\m$\,. 
Let us denote the left and right actions of $\G$ on itself by $l_y(x):=yx$ and $r_y(x):=xy^{-1}$\,.
They induce actions on the $C^*$-algebra $\mathcal C_b(\G)$ of all bounded continuous complex functions on $\G$\,. To get pointwise continuous actions we restrict, respectively, to bounded left and right uniformly continuous functions:
\begin{equation*}\label{lefd}
{\sf l}:\G\to{\sf Aut}\big[\mathcal{LUC}(\G)\big]\,,\quad\big[{\sf l}_y(c)\big](x):=\big(c\circ l_{y^-1}\big)(x)=c(y^{-1}x)\,,
\end{equation*}
\begin{equation*}\label{righd}
{\sf r}:\G\to{\sf Aut}\big[\mathcal{RUC}(\G)\big]\,,\quad\big[{\sf r}_y(c)\big](x):=\big(c\circ r_{y^{-1}}\big)(x)=c(xy)\,.
\end{equation*}
We denoted by ${\sf Aut}(\D)$ the group of all the automorphisms of the $C^*$-algebra $\D$ with the strong (pointwise convergence) topology.

A $C^*$-subalgebra $\A$ of $\mathcal{LUC}(\G)$ is called {\it left-invariant} if ${\sf l}_y\A\subset\A$ for every $y\in\G$\,; {\it right-invariance} is defined analogously for $C^*$-subalgebras of $\mathcal{RUC}(\G)$\,. Then $\mathcal C_0(\G)$\,, the $C^*$-algebra of continuous complex functions which decay at infinity, is an invariant ideal, both to the left and to the right.

If ${\sf K}$ is a closed subgroup of $\G$ then $\G/{\sf K}:=\{x{\sf K}\mid x\in\G\}$ denotes the locally compact space of orbits of the action ${\sf r}$ restricted to ${\sf H}$\,.  One denotes by $p_{\sf K}:\G\rightarrow\G/{\sf K}$ the canonical surjection. By the formula $P_{\sf K}(\tilde c):=\tilde c\circ p_{\sf K}$ it can be used to transform functions on $\G/{\sf K}$ into functions on $\G$ which are right-invariant under ${\sf K}$\,, meaning that ${\sf r}_y(c)=c$ for every $y\in{\sf K}$\,.
Conversely, such a right-invariant function $c$ can be written in the form $c=P_{\sf K}(\tilde c)$ for a unique function $\tilde c:=c_{\sf K}$ defined on the quotient by $c_{\sf K}(x{\sf K}):=c(x)$\,.

For a $C^*$-subalgebra $\mathcal D$ of $\mathcal{LUC}(\G)$ and a compact subgroup ${\sf K}$ of $\G$ one defines the ${\sf K}$-fixed point $C^*$-algebra
\begin{equation*}\label{fiert}
\mathcal D^{({\sf K})}:=\{a\in\mathcal D\mid {\sf r}_y(a)=a\,,\,\forall\,y\in{\sf K}\}=\mathcal D\cap\mathcal{LUC}(\G)^{({\sf K})}
\end{equation*}
and the ideal $\mathcal D_0:=\mathcal D\cap\mathcal C_0(\G)$\,. If $\mathcal D$ is left invariant, both $\mathcal D^{({\sf K})}$ and $\mathcal D_0$ are left invariant.

We call {\it $\sf G$-compactification of $\,\G/{\sf K}$} a compact dynamical system $(\O,\vartheta,\G)$ and an embedding of $\,\G/{\sf K}$ as a dense open $\vartheta$-invariant subset of $\,\O$ such that on $\G/{\sf K}$ the action $\vartheta$ coincides with the natural action of $\,\G$ given by $\,\vartheta_x(y{\sf K}):=(xy){\sf K}=l_x(y){\sf K}$\,.

\begin{Definition}\label{stufat}
{\rm A right invariance data} is a triple $\Upsilon=({\sf K},\Omega,\vartheta)$ where ${\sf K}$ is a compact subgroup of $\,\G$ and $(\Omega,\vartheta,\G)$ is a  $\sf G$-compactification of $\,\G/{\sf K}$\,.
\end{Definition}

The right invariance data being given, let us set $\,\mathcal C\<\O\>:=\big\{b|_{\G/{\sf K}}\!\mid\!b\in\mathcal C(\O)\big\}$\,. It is a left invariant $C^*$-subalgebra of $\mathcal{LUC}(\G/{\sf K})$ (left uniform continuity refers to the action of $\sf G$ in $\G/{\sf K}$\,), containing $\mathcal C_0(\G/{\sf K})$ and isomorphic to $\mathcal C(\O)$\,. One sets $\Omega_\infty:=\O\setminus(\G/{\sf K})$\,; it is the space of a compact dynamical system $\big(\O_\infty,\vartheta,\G)$\,.  The quotient of $\mathcal C\<\O\>$ by the ideal $\mathcal C_0(\G/{\sf K})$ is canonically isomorphic to $\mathcal C(\O_\infty)$\,. 

To each right invariance data $\Upsilon:=({\sf K},\O,\vartheta)$ one associates
\begin{equation}\label{spalat}
\A[{\sf K},\Omega]:=\big\{a\in\mathcal{LUC}(\G)^{({\sf K})}\,\big\vert\, a_{\sf K}\in\mathcal C\<\O\>\big\}\,.
\end{equation}
It is formed of bounded left uniformly continuous functions on $\G$ which are ${\sf K}$-fixed points to the right and which, after reinterpretation,  can be extended continuously to the compactification $\Omega$ of $\G/{\sf K}$\,. 

\begin{Lemma}\label{stors}
$\A[{\sf K},\Omega]$ is a left invariant unital $C^*$-subalgebra of $\mathcal{LUC}(\G)$ equivariantly isomorphic to $\mathcal C\<\O\>$\,. Its quotient by $\A[{\sf K},\Omega]_0$ is equivariantly isomorphic to $\mathcal C\big(\O_\infty\big)$\,.
\end{Lemma}

\begin{proof}
Checking that $\A[{\sf K},\Omega]$ is a unital left invariant $C^*$-subalgebra is straightforward. 
The isomorphism is $\,a\to a_{\sf K}$\,. In addition, if $x{\sf K}\in\G/{\sf K}$\,, then 
$$
({\sf l}_z a)_{\sf K}(x{\sf K})=({\sf l}_z a)(x)=a\big(z^{-1}x\big)=a_{\sf K}\big(z^{-1}x{\sf K}\big)=a_{\sf K}\big[\vartheta_{z^{-1}}(x{\sf K})\big]=:\big[\theta_z(a_{\sf K})\big](x{\sf K})\,,
$$
which proves the equivariance. For the quotient, note that 
\begin{equation*}\label{prajit}
\A[{\sf K},\Omega]_0=\big\{a\in\mathcal C_0(\G)^{({\sf K})}\,\big\vert\,a_{\sf K}\in\mathcal C\<\O\>\big\}=\mathcal C_0(\G)^{({\sf K})}\cong\mathcal C_0(\G/{\sf K})
\end{equation*}
is independent of $\O$\,. The second equality follows from the fact that any $\mathcal C_0(\Si)$-function on a locally compact space $\Si$ extends continuously to any compactification of $\Si$\,. Then, equivariantly, one has
$$
\A[{\sf K},\Omega]/\A[{\sf K},\Omega]_0\cong\mathcal C(\Omega)/\mathcal C_0(\G/{\sf K})\cong\mathcal C\big(\O_\infty\big)\,.
$$
\end{proof}

Taking ${\sf K}=\{\e\}$ one gets the most general left-invariant $C^*$-subalgebra $\A[\{\e\},\Omega]:=\big\{a\in\mathcal{LUC}(\G)\mid a\in\mathcal C\<\O\>\big\}$ of $\mathcal{LUC}(\G)$ containing $\mathcal C_0(\G)$ and perhaps this is the most interesting situation. Here $\O$ is just a $\G$-compactification of $\G$\,. In particular $\mathcal{LUC}(\G)=\A\big[\{\e\},\Omega^{\rm luc}\big]$\,, where $\Omega^{\rm luc}$ is the (left) uniform compactification of $\G$\,.

\begin{Remark}\label{baimari}
{\rm Actually the $C^*$-algebra $\A[{\sf K},\Omega]$ associated to some right invariance data $\Upsilon:=({\sf K},\O,\vartheta)$ is the most general unital left-invariant $C^*$-subalgebra $\A$ of $\mathcal{LUC}(\G)$ for which $\A_0:=\A\cap\mathcal C_0(\G)$ is an essential ideal. 

To see this, we note that $\A_0$ is a left invariant $C^*$-subalgebra of $\mathcal C_0(\G)$\,. Let us recall \cite[Lemma 12]{LL} that the non-trivial left-invariant $C^*$-subalgebras of $\mathcal C_0(\G)$ are in one-to-one correspondence with the compact subgroups of $\G$\,. Explicitly, applying this to $\mathcal A_0$\,,
\begin{equation*}\label{lau}
{\sf K}_{\mathcal A}:=\big\{\,y\in\G\mid {\sf r}_y(a)=a\,,\ \forall\,a\in\mathcal A_0\,\big\}
\end{equation*}
is a compact subgroup of $\G$ and one has the fixed-point characterization
\begin{equation*}\label{deleew}
\mathcal A_0=\{\,a\in\mathcal C_0(\G)\mid {\sf r}_y(a)=a\,,\ \forall\,y\in{\sf K}_\mathcal A\,\}=:\mathcal C_0(\G)^{({\sf K_\A})}\cong\mathcal C_0\big(\G/{\sf K}_{\mathcal A}\big)\,.
\end{equation*} 
Then one uses the fact that all the essential unitizations of $\A_0$ are given by compactifications $\Omega$ of $\G/{\sf K}_\A$ and that left invariance forces $\O$ to be a $\G$-compactification for some action $\vartheta$\,.
}
\end{Remark}

\section{Crossed product $C^*$-algebras}\label{flekarin}

We recall briefly some basic facts about crossed products, referring to \cite{Wi} for a full treatment.

The basic data is a {\it $C^*$-dynamical system}  $(\A,\alpha,\G)$\,, consisting of a strongly continuous action $\alpha$ of the locally group $\G$ by automorphisms of the $C^*$-algebra $\G$\,. For us, $\A$ will always be commutative.
To the $C^*$-dynamical system $(\A,\alpha,\G)$ we associate  $L^1(\G;\A)$ (the space of all Bochner integrable functions $\G\to\A$ with the obvious norm) with the Banach $^*$-algebra structure given by
\begin{equation*}\label{diamond}
(\Phi\diamond \Psi)(x):=\int_{\G}\!\Phi(y)\,\alpha_y\!\left[\Psi(y^{-1}x)\right]d\m(y)\,,
\end{equation*}
\begin{equation*}\label{willibrant}
\Phi^\diamond(x):=\Delta(x)^{-1}\,\alpha_x\!\left[\Phi(x^{-1})^*\right]\,.
\end{equation*}
{\rm The crossed product $C^*$-algebra} $\A\!\rtimes_\alpha\!\G$ is the enveloping $C^*$-algebra of this Banach $^*$-algebra, i.e its completion in the universal norm $\p\!\Phi\!\p_{\rm univ}\,:=\,\sup_{\Pi}\p\!\Pi(\Phi)\!\p_{\mathbb B(\H)}$\,,
with the supremum taken over the non-degenerate $^*$-representations $\,\Pi:L^1(\G;\A)\rightarrow\mathbb B(\H)$\,. 
The Banach space $L^1(\G;\A)$ can be identified with the projective tensor product $\A\,\overline\otimes\,L^1(\G)$\,, and $\mathcal C_{\rm c}(\G;\A)$\,, the space of all $\A$-valued continuous compactly supported function on $\G$\,, is a dense $^*$-subalgebra. 

{\it A covariant representation} of the $C^*$-dynamical system $(\A,\alpha,\G)$  is a triple $(\rho,T,\H)$ where $\H$ is a Hilbert space, $\rho:\A\rightarrow\mathbb B(\H)$ is a non-degenerate $^*$-representation,
$T:\G\rightarrow\mathbb U(\H)$ is strongly continuous and satisfies
\begin{equation*}\label{clasa}
T(x)T(y)=T(xy)\,,\quad\forall\,x,y\in\G\,,
\end{equation*}
\begin{equation*}\label{holds}
T(x)\rho(a)T(x)^*=\rho\left[\alpha_x(a)\right]\,,\quad\forall\,a\in\A\,,\,x\in\G\,.
\end{equation*}
{\it The integrated form} of the covariant representation $(\rho,T,\H)$ is the unique continuous extension $\,\rho\!\rtimes\!T:\A\!\rtimes_\alpha\!\G\rightarrow\mathbb B(\H)$ defined initially on $L^1(\G;\A)$ by
\begin{equation}\label{maximilian}
(\rho\!\rtimes\!T)(\Phi):=\!\int_\G\!\rho[\Phi(x)]T(x)d\m(x)\,.
\end{equation}

\begin{Remark}\label{mamsatu}
{\rm For further use, let us also examine $^*$-morphisms in the setting of crossed products (cf \cite{Wi}). Assume that $(\A,\alpha,\G)$ and 
$(\A',\alpha',\G)$ are $C^*$-dynamical systems and $\pi:\A\rightarrow\A'$ is {\it an equivariant $^*$-morphism}, i.e. a $^*$-morphism satisfying 
$\pi\circ\alpha_x=\alpha'_x\circ\pi\,,\ \forall\,x\in\G$\,. One defines 
\begin{equation}\label{varzar}
\pi^\rtimes:L^1(\G;\A)\rightarrow L^1(\G;\A')\,,\quad\big[\pi^\rtimes(\Phi)\big](x):=\pi[\Phi(x)]\,.
\end{equation}
It is easy to check that $\pi^\rtimes$ is a $^*$-morphism of the two Banach $^*$-algebra structures and thus it extends to a $^*$-morphism 
$\pi^\rtimes:\A\!\rtimes_\alpha\!\G\rightarrow\A'\!\rtimes_{\alpha'}\!\G$\,. If $\pi$ is injective, $\pi^\rtimes$ is also injective.
}
\end{Remark}

Our most important crossed product will be attached to the $C^*$-dynamical system $\big(\A[{\sf K},\O],{\sf l},\G\big)$\,, where $\A[{\sf K},\O]$ is  the left invariant $C^*$-subalgebra of $\mathcal{LUC}(\G)$ associated as in Section \ref{firica} to a right invariance data $\Upsilon=({\sf K},\O,\vartheta)$\,.
For $\A[{\sf K},\O]$-valued functions $\Phi$ defined on $\G$ and for elements $x,q$ of the group, we are going to use notations as $[\Phi(x)](q)=:\Phi(q;x)$\,, interpreting $\Phi$ as a function of two variables. For the reader's convenience, we rewrite the general formulae defining the twisted crossed product in a concrete form:
\begin{equation*}\label{diamond}
(\Phi\diamond \Psi)(q;x):=\int_{\G}\!\Phi(q;y)\,\Psi\big(y^{-1}q;y^{-1}x\big)\,d\m(y)\,,
\end{equation*}
\begin{equation*}\label{willibrant}
\Phi^\diamond(q;x):=\Delta(x)^{-1}\,\overline{\Phi(x^{-1}q;x^{-1})}\,.
\end{equation*}

\medskip
In this case we have a natural covariant representation $\big({\sf Mult},T,\H\big)$\,, for historical reasons called {\it the Schr\"odinger representation}, given in $\mathcal H:=L^2(\G)$ by
\begin{equation*}\label{razvan}
\[T(y)u\]\!(q):=u\!\(y^{-1}q\)\,,\ \quad {\sf Mult}(a)u:=a u\,.
\end{equation*}
The corresponding integrated form ${\sf Sch}:={\sf Mult}\rtimes T$ is given for $\Phi\in L^1\big(\G;\A[{\sf K},\O]\big)$ and $u\in L^2(\G)$ by 
\begin{equation}\label{rada}
\[{\sf Sch}(\Phi)u\]\!(q)=\!\int_\G\Phi\big(q;z\big)u(z^{-1}q)\,d\m(z)=\!\int_\G\Phi\!\left(q;qy^{-1}\right)\!u(y)\Delta(y)^{-1}d\m(y)\,. 
\end{equation}

\begin{Proposition}\label{danda}
The $C^*$-algebra 
$\mathfrak C[{\sf K},\O]:={\sf Sch}\big(\A[{\sf K},\O]\rtimes_{\sf l}\G\big)$ 
is isomorphic to the reduced crossed product $\big(\A[{\sf K},\O]\!\rtimes_{\sf l}\!\G\big)_{\rm red}$\,. 
If $\,\G$ is amenable, the representation 
$$
{\sf Sch}:\A[{\sf K},\O]\rtimes_{\sf l}\G\rightarrow\mathfrak C[{\sf K},\O]\subset\mathbb B\big[L^2(\G)\big]
$$ 
is faithful.
\end{Proposition}

\begin{proof}
This follows from the fact that, up to multiplicity, the Schr\"odinger representation is unitarily equivalent to a left regular representation of the full crossed product: ${\sf Reg}\cong 1_{L^2(\G)}\otimes{\sf Sch}$\,. For the details see \cite[Prop. 7.9]{MR} for instance.
\end{proof}

\begin{Remark}\label{fredegonda}
{\rm Let us define the convolution operator ${\sf Conv}:L^1(\G)\rightarrow\mathbb B\!\[L^2(\G)\]$ by
\begin{equation*}
\big[{\sf Conv}(\varphi)u\big](q)=\[{\sf Sch}(1\otimes\varphi)u\]\!(q)=\int_\G\!\varphi\!\left(qy^{-1}\right)\!u(y)\Delta(y)^{-1}d\m(y)\,.
\end{equation*}
Setting $(a\otimes\varphi)(q;x):=a(q)\varphi(x)$\,, one gets immediatly a product of a multiplication operator with a convolution operator
\begin{equation*}\label{crocodil}
{\sf Sch}(a\otimes\varphi)={\sf Mult}(a)\,{\sf Conv}(\varphi)\,.
\end{equation*}

The represented $C^*$-algebra $\,\mathfrak C[{\sf K},\O]:={\sf Sch}\big(\A[{\sf K},\O]\rtimes_{\sf l}\G\big)$ coincides with the closed vector space spanned by products of the form ${\sf Mult}(a)\,{\sf Conv}(\varphi)$ with $a\in\A[{\sf K},\O]$ and  $\varphi\in L^1(\G)$\,. It is enough to use functions $\,\varphi\in\mathcal C_{\rm c}(\G)$ which are continuous and compactly supported.
}
\end{Remark}

\begin{Example}\label{ospat}
{\rm If the group $\G$ is discrete, then $\,\mathcal{LUC}(\G)=\mathcal{RUC}(\G)=\ell^\infty(\G)$\,. The crossed product $\ell^\infty(\G)\!\rtimes_{\sf l}\G$ is called traditionally {\it the $C^*$-algebra of band dominated operators} or sometimes, in the context of coarse geometry, {\it the uniform  Roe algebra}. It contains a lot of interesting subclasses, as described in \cite{CWL,Lin,RRS,Wtt} for example.
In this case $\delta_{\e}\in \ell^1(\G)$ and the multiplication operator ${\sf Mult}(a)={\sf Sch}(a\otimes\delta_{\e})$ also belongs to the $C^*$-subalgebra ${\sf Sch}\big[\ell^\infty(\G)\!\rtimes_{\sf l}\G\big]$ of $\mathbb B\big[\ell^2(\G)\big]$\,.
So, besides products of the form ${\sf Mult}(a){\sf Conv}(\varphi)$\,,  perturbations 
\begin{equation*}\label{convmult}
{\sf Sch}\big(1\otimes\varphi+a\otimes\delta_{\e}\big)={\sf Conv}(\varphi)+{\sf Mult}(a)
\end{equation*} 
of convolution operators by operators of multiplication with elements of $\ell^\infty(\G)$ also belong to our $C^*$-algebra. These are a reminiscence of usual Schr\" odinger operators.
}
\end{Example}

\section{Pseudodifferential operators on type I groups}\label{furitar}

To switch to a realization fitted to studying pseudo-differential operators \cite{MR}, we need more assumptions on the group $\G$\,, allowing a manageable Fourier transformation.

We set $\,\wG:={\rm Irrep(\G)}/_{\cong}$ and call it {\it the unitary dual of $\,\G$}\,; by definition, it is composed of unitary equivalence classes of strongly continuous irreducible Hilbert space representation $\pi:\G\rightarrow\mathbb U(\H_\pi)\subset\mathbb B(\H_\pi)$.  There is a standard Borel structure on $\wG$\,, called {\it the Mackey Borel structure} \cite[18.5]{Di}. 
If $\G$ is Abelian, the unitary dual $\wG$ is the Pontryagin dual group; if not, $\wG$ has no group structure.

We denote by $C^*(\G)$ the full (universal) $C^*$-algebra of $\G$\,. 
Any representation $\pi$ of $\G$ generates canonically a non-degenerate represention $\Pi$ of the $C^*$-algebra $C^*(\G)$\,.

\begin{Definition}\label{uja}
The locally compact group $\G$ is {\it type I} if for every irreducible representation $\pi$ one has $\,\mathbb K(\H_\pi)\subset\Pi\big[C^*(\G)\big]$\,.
It will be called {\rm admissible} if it is second countable, type I and unimodular.
\end{Definition}

For  the concept of {\it type I group} and for examples we refer to \cite{Di,Fo1}; a short summary can be  found in \cite[Sect. 2]{MR}.  In \cite[Th. 7.6]{Fo1} (see also \cite{Di}), many equivalent characterisations are given for a second  countable locally compact group to be type I.  The main consequence of this property is the existence of a measure on the unitary dual $\wG$ for which a Plancherel Theorem holds.
This is a measure on $\wG$\,, called {\it the Plancherel measure associated to $\m$} and denoted by $\wm$ \cite[18.8]{Di}. It plays a basic role in defining a Fourier transform.

It is known that there is a $\wm$-measurable field $\big\{\,\H_\xi\mid\xi\in\wG\,\big\}$ of Hilbert spaces  and a measurable section $\wG\ni\xi\mapsto\pi_\xi\in{\rm Irrep(\G)}$ such that each $\pi_\xi:\G\rightarrow\mathbb B(\H_\xi)$ is a irreducible representation belonging to the class $\xi$\,. By a systematic abuse of notation, instead of $\pi_\xi$ we will write $\xi$\,, identifying irreducible representations (corresponding to the measurable choice) with elements of $\wG$\,.  

The Fourier transform \cite[18.2]{Di} of $u\in L^1(\G)$ is given in weak sense by
\begin{equation*}\label{ion}
({\fscr F}u)(\xi)\equiv \widehat{u}(\xi):=\int_\G u(x)\xi(x)^*d\m(x) \in\mathbb B(\H_\xi)\,.
\end{equation*} 
It defines an injective linear contraction ${\fscr F}:L^1(\G)\rightarrow \mathscr B(\wG)$\,, where $\mathscr B(\wG):=\int^\oplus_{\wG}\mathbb B(\H_\xi)d\wm(\xi)$ is a direct integral von Neumann algebra. One also introduces the direct integral Hilbert space 
\begin{equation*}\label{andy}
\mathscr B^2(\wG):=\int_{\wG}^\oplus\!\B^2(\H_\xi)\,d\wm(\xi)\,\cong\int_{\wG}^\oplus\!\H_\xi\otimes\overline\H_\xi\,d\wm(\xi)\,,
\end{equation*}
with the scalar product
\begin{equation*}\label{justin}
\<\phi_1,\phi_2\>_{\mathscr B^2(\wG)}:=\int_{\wG}\,\<\phi_1(\xi),\phi_2(\xi)\>_{\mathbb B^2(\H_\xi)}d\wm(\xi)=\int_{\wG}{\rm Tr}_\xi\!\[\phi_1(\xi)\phi_2(\xi)^*\]d\wm(\xi)\,,
\end{equation*}
where ${\rm Tr}_\xi$ is the trace in $\mathbb B(\H_\xi)$\,.
A generalized form of Plancherel's Theorem \cite{Di,Fo1} states that {\it the Fourier transform ${\fscr F}$ extends from $L^1(\G)\cap L^2(\G)$ to a unitary isomorphism ${\fscr F}:L^2(\G)\rightarrow \mathscr B^2(\wG)$}\,. 

Let us come back to the left invariant algebra $\A[{\sf K},\O]$ associated to the right invariance data $({\sf K},\O)$\,. We already mentioned that $L^1\big(\G;\A[{\sf K},\O]\big)$ can be identified with the completed projective tensor product $\A[{\sf K},\O]\,\overline\otimes\,L^1(\G)$\,. Then one gets a linear continuous injection 
\begin{equation*}\label{absenta}
{\sf id}\,\overline\otimes\,{\fscr F}:\A[{\sf K},\O]\,\overline\otimes\,L^1(\G)\rightarrow\A[{\sf K},\O]\,\overline\otimes\,\mathscr B(\wG)
\end{equation*}
and we endow the image $\big({\sf id}\,\overline\otimes\,{\fscr F}\big)\big(\A[{\sf H},\O]\,\overline\otimes\,L^1(\G)\big)$ with the Banach $^*$-algebra structure transported from $L^1\big(\G;\A[{\sf H},\O]\big)\cong \A[{\sf H},\O]\,\overline\otimes\,L^1(\G)\,$ through $\,{\sf id}\,\overline\otimes\,{\fscr F}$\,. 

Let us denote by $\mathfrak A[{\sf H},\O]$ the crossed product $\A[{\sf K},\O]\!\rtimes_{\sf l}\!\G$\,, which is the envelopping $C^*$-algebra  of the Banach $^*$-algebra $L^1\big(\G;\A[{\sf K},\O]\big)$. Similarly, we denote by $\mathfrak B[{\sf K},\O]$ the envelopping $C^*$-algebra of the Banach $^*$-algebra $({\sf id}\,\overline\otimes\,{\fscr F})\big(\A[{\sf K},\O]\,\overline\otimes\,L^1(\G)\big)$\,. By the universal property of the enveloping functor, the map ${\sf id}\,\overline\otimes\,{\fscr F}$ extends to an isomorphism $\,\mathfrak F:\mathfrak A[{\sf K},\O]\rightarrow\mathfrak B[{\sf K},\O]$\,. 

If we compose the Schr\"odinger representation ${\sf Sch}$ with the inverse of this partial Fourier transform one finds the pseudo-differential representation
\begin{equation}\label{surprize}
{\sf Op}={\sf Sch}\circ\mathfrak F^{-1}:\mathfrak B[{\sf K},\O]\to\mathbb B\big[L^2(\G)\big]\,,
\end{equation}
given explicitly (on reasonable symbols $f$) by
\begin{equation}\label{bilfret}
\[{\sf Op}(f)u\]\!(x)=\int_\G\!\int_{\wG}\,{\rm Tr}_\xi\Big[\xi(xy^{-1})f\big(x,\xi\big)\Big]u(y)d\m(y)d\wm(\xi)\,.
\end{equation}
One could say roughly that ${\sf Op}(f)$ is {\it a strictly negative order pseudodifferential operator on $\G$ with coefficients in the $C^*$-algebra $\A[{\sf K},\O]$}\,. Note that its symbol $f$ is globally defined and, for non-commutative groups, it is operator-valued. If $\G=\R^n$ one recovers the Kohn-Nirenberg quantization.

More information on these type of operators, including motivations, extension to distributions,  $\tau$-quantizations (allowing a Weyl symmetric form in some cases) can be found in \cite{MR}. For compact Lie groups we refer to \cite{RT,RT1,RTW} (see also references therein), while the nilpotent case is treated in \cite{FR,FR1}.

\section{The main result}\label{flokoshin}

Al over this section a right invariance data $\Upsilon=({\sf K},\O,\vartheta)$ will be fixed. Associated to it one has the left invariant subalgebra $\A[{\sf K},\O]$ of $\mathcal{LUC}(\G)$ and the crossed product $\,\mathfrak A[{\sf K},\O]:=\A[{\sf K},\O]\!\rtimes_{\sf l}\!\G$ represented by the $C^*$-subalgebra $\mathfrak C[{\sf K},\O]$ of $\mathbb B\big[L^2(\G)\big]$\,. 

If the group $\G$ is admissible, one also has the partially Fourier transformed $C^*$-algebra $\,\mathfrak B[{\sf K},\O]:=\mathfrak F\,\mathfrak A[{\sf K},\O]$\,, composed essentially by symbols of pseudo-differential operators, and the following diagram commutes:
$$
\begin{diagram}
\node{\mathfrak A[{\sf K},\O]} \arrow{e,t}{\mathfrak F}\arrow{s,r}{\sf Sch}\node{\mathfrak B[{\sf K},\O]}\arrow{sw,r}{{\sf Op}}\\ 
 \node{\mathfrak C[{\sf K},\O]}
\end{diagram}
$$

We set $\,\Omega_\infty:=\O\setminus(\G/{\sf K})$ and recall that $\big(\Omega_\infty,\vartheta,\G\big)$ is a compact dynamical system giving rise to the $C^*$-dynamical system $\big(\mathcal C(\O_\infty),\th,\G\big)$\,. By definition, {\it a quasi-orbit} is the closure of an orbit. So any point $\o\in\O_\infty$ generates an orbit ${\sf O}^\o:=\vartheta_\G(\o)$ and a (compact) quasi-orbit  ${\sf Q}^\o:=\overline{{\sf O}^\o}$\,.

For every closed subset $F$ of $\O_\infty$ we denote by $\mathcal C^F\!\big(\O_\infty\big)$ the ideal of all the elements $b$ of $\mathcal C(\O_\infty)$ such that $b|_F=0$\,. The quotient $\mathcal C\big(\O_\infty\big)/\mathcal C^F\!\big(\O_\infty\big)$ is naturally isomorphic to $\mathcal C(F)$\,. If $F$ is also $\vartheta$-invariant, this gives rise \cite[Prop. 3.19]{Wi} to a canonical isomorphism at the level of crossed products: $\mathcal C\big(\O_\infty\big)\!\rtimes\!\G\,/\,\mathcal C^F\!\big(\O_\infty\big)\!\rtimes\!\G\cong \mathcal C(F)\!\rtimes\!\G$ (the three actions are connected to the initial action $\th$ in an obvious way). The case $F={\sf Q}^\o$ will play an important role in our arguments below.

Let $\Phi\in L^1\big(\G;\A[{\sf K},\O]\big)$ and $\o\in\O_\infty$\,. We indicate now a procedure to associate to this function $\Phi:\G\times\G\to\mathbb C$ another function $\Phi^{\o}:\G\times\G\to\mathbb C$\,. The next constructions are done up to negligible subsets in the second variable $x$\,; one could also start with regular elements $\Phi\in \mathcal C_{\rm c}\big(\G;\A[{\sf K},\O]\big)$ and then invoke exensions.

\begin{enumerate}
\item[(i)]
Define $\,\Phi_{\sf K}:\G/{\sf K}\times\G\to\mathbb C$ by $\,\Phi_{\sf K}(q{\sf K},x):=\Phi(q,x)$\,.
\item[(ii)]
Extend $\Phi_{\sf K}$ to a function $\widetilde\Phi_{\sf K}$ on $\O\times\G$ continuous in the first variable.
\item[(iii)]
Define $\Phi^\o$ by $\Phi^\o(q,x):=\widetilde\Phi_{\sf K}(\vartheta_q(\o),x)$\,.
\end{enumerate}

Note that (i) is connected to the fixed-point condition in the definition of $\A[{\sf K},\O]$ and (ii) to the second half of its definition, describing the behaviour at infinity of its elements by the fact that, after reinterpretation, they can be extended continuously on $\O=(\G/{\sf K})\sqcup\O_\infty$\,.

A rough way to describe this procedure is to say that $\Phi$ is first transformed in a function on the quotient (this step is not necessary if ${\sf H}=\{\e\}$), then it is extended to $\O\times\G$ and restricted to ${\sf O}^\o\times\G$\,, and then the variable along the orbit is reinterpreted as the variable in the group.

It is easy to see that $\Phi^\o\in L^1\big(\G;\mathcal{LUC}(\G)\big)$\,. 
So, if we represent $\Phi$ as an operator (a Hamiltonian) $H:={\sf Sch}(\Phi)$ in $L^2(\G)$\,, we also get a family of "asymptotic Hamiltonians" 
\begin{equation}\label{simptotik}
\big\{H^\o\!:={\sf Sch}(\Phi^\o)\,\big\vert\, \o\in\O_\infty\big\}\,.
\end{equation}
Explicitly, for $u\in L^2(\G)$\,, one has by \eqref{rada}
\begin{equation}\label{explicit}
\big[H^\o (u)\big](q)=\int_\G\Phi\!\left(\vartheta_q(\o),qy^{-1}\right)\!u(y)\,d\m(y)\,.
\end{equation}

\begin{Remark}\label{eventual}
{\rm It will be relevant for our Theorem \ref{labba} to understand in advance the dependence of $H^\o$ on $\o$\,. If $\o$ and $\o'$ belong to the same orbit, then $H^\o$ and $H^{\o'}$ are unitary equivalent (use an element $x\in\G$ such that $\vartheta_x(\o)=\o'$ to build the unitary operator). On the other hand, two points $\o,\o'$ could generate the same quasi-orbit (i.e. $\overline{{\sf O}^\o}=\overline{{\sf O}^{\o'}}$\,) without belonging to the same orbit (think of minimal systems, for example). In such a situation it is still true that ${\rm sp}\big(H^{\o}\big)={\rm sp}\big(H^{\o'}\big)$\,. We refer to \cite[Sect. 7.4]{MR} for proofs (in a slightly different context), but the same conclusions will also follow from our proof in Section \ref{furitut}.
}
\end{Remark}

\begin{Definition}\label{ghiwech} 
A set $\{\o_i\mid i\in I\}$ of points of $\,\O_\infty$ is called {\rm a sufficient family} if the associated quasi-orbits $\big\{{\sf Q}^{\o_i}\mid i\in I\big\}$ form a covering of $\,\O_\infty$\,, i.e. $\bigcup_{i\in I}{\sf Q}^{\o_i}=\O_\infty$\,.
\end{Definition}

We adapt now \cite[Def. 6.17]{Wi} to the case of an Abelian $C^*$-algebra $\,\A=\mathcal C(\O_\infty)$\,. Recall that every representation of a crossed product is deduced from a covariant representation.

\begin{Definition}\label{fromthere}
A representation $P =p\rtimes U$ of $\,\mathcal C(\O_\infty)\!\rtimes_\th\!\G$ {\rm lives on a quasi-orbit} if there exists $\,\o\in\O_\infty$ such that $\,{\rm Res}(\ker P):=\ker(p)=\mathcal C^{\sf Q^\o}\!(\O_\infty)$\,.

The dynamical system $\big(\O_\infty,\vartheta,\G\big)$ is called {\rm quasi-regular} if every irreducible representation of $\,\mathcal C(\O_\infty)\!\rtimes_{\th}\!\G$ lives on a quasi-orbit.
\end{Definition}

Let us now state our main result.

\begin{Theorem}\label{labba}
Assume that the locally compact group $\G$ is amenable, let $\Upsilon=\big({\sf K}\,,\,\Omega=(\G/{\sf K})\sqcup\O_\infty,\vartheta\big)$ be a right invariance data and assume that $\big(\O_\infty,\vartheta,\G\big)$ is quasi-regular. Let $\Phi\in L^1\big(\G;\A[{\sf K},\O]\big)$ and $H:={\sf Sch}(\Phi)\in \mathfrak C[{\sf K},\O]\subset\mathbb B\big[L^2(\G)\big]$\,. Let $\{\o_i\mid i\in I\,\}$ a sufficient family of points of $\,\O_\infty$ and for each $i\in I$ set $H^{\o_i}\!:={\sf Sch}\big(\Phi^{\o_i}\big)$\,.  
\begin{enumerate}
\item
One has
\begin{equation}\label{fich}
{\rm sp}_{\rm ess}(H)=\underset{i\in I}{\bigcup}\,{\rm sp}\big(H^{\o_i}\big)\,.
\end{equation}
\item
The operator $H$ is Fredholm if and only if every $H^{\o_i}$ is invertible.
\end{enumerate}
\end{Theorem}

\begin{Remark}\label{ginaml}
{\rm One can even take $\Phi\in\mathfrak A[{\sf K},\O]$\,, but the operations leading from $\Phi$ to $\Phi^{\o_i}$ must be understood, after extension, as abstract $C^*$-morphisms; even in simple situations some of the elements of the crossed products are no longer usual functions.}
\end{Remark}

\begin{Remark}\label{ginall}
{\rm If $\G$ is admissible, we can recast Theorem \ref{labba} in the equivalent language of pseudo-differential operators introduced in Section \ref{furitar}. One just has to express $H$ as ${\sf Op}(f)$ for an operator symbol $f=\mathfrak F \Phi$ of $\mathfrak B({\sf K},\O)$ and then gets symbols $f^{\o_i}$ such that $H^{\o_i}={\sf Op}\big(f^{\o_i}\big)$ by performing suitable operations on $f$ or, equivalently, by setting $f^{\o_i}=\mathfrak F \Phi^{\o_i}$ for every $i\in I$\,.
We leave the details to the reader.}
\end{Remark}

\begin{Remark}\label{gental}
{\rm The result describing the essential spectrum can be extended to unbounded operators affiliated to $\mathfrak C[{\sf K},\O]\subset\mathbb B\big[L^2(\G)\big]$\,, i.e. self-adjoint densely defined operators in $L^2(\G)$ for which the resolvent family belongs to $\mathfrak C[{\sf K},\O]$\,. For this one has to use \cite[Sect. 5]{NP} (see also \cite{GI1,GI2,Ma,MPR2} and references therein for previous work). Affiliation to crossed product $C^*$-algebras with non-commutative groups is a difficult issue; hopefully a future publication will be devoted to this issue and to its spectral consequences.
}
\end{Remark}

\section{Proof of Theorem \ref{labba}}\label{furitut}

Our operator $H$ belongs to the $C^*$-subalgebra $\mathfrak C[{\sf K},\O]$ of $\mathbb B\big[L^2(\G)\big]$\,. It is well-known that its essential spectrum coincides with the usual spectrum of its image in the quotient $C^*$-algebra $\mathfrak C[{\sf K},\O]\,/\,\mathfrak C[{\sf K},\O]_0$\,, where we set $\mathfrak C[{\sf K},\O]_0:=\mathfrak C[{\sf K},\O]\cap\mathbb K\big[L^2(\G)\big]$\,. (This quotient can be regarded as a $C^*$-subalgebra of the Calkin algebra.) In addition, by Atkinson's Theorem, $H$ is Fredholm if and only if its canonical image in the quotient is invertible.

It is also well-known \cite{Wi} that the Schr\"odinger $^*$-representation sends $\mathcal C_0(\G)\!\rtimes_{\sf l}\!\G$ onto $\mathbb K\!\[L^2(\G)\]\subset\mathbb B\!\[L^2(\G)\]$\,; it is an isomorphism between $\mathcal C_0(\G)\!\rtimes_{\sf l}\!\G$ and $\mathbb K\!\[L^2(\G)\]$ since $\G$ is amenable.
Taking this into account, we see that $\,{\sf Sch}:\mathfrak A[{\sf K},\O]\to\mathfrak C[{\sf K},\O]$ is an isomorphism sending $\mathfrak A[{\sf K},\O]_0:=\mathfrak A[{\sf K},\O]\cap\big[\mathcal C_0(\G)\!\rtimes_{\sf l}\!\G\big]$ into $\mathfrak C[{\sf K},\O]_0$\,. This leads to isomorphisms
\begin{equation}\label{lasat}
\mathfrak A[{\sf K},\O]\,/\,\mathfrak A[{\sf K},\O]_0\cong\mathfrak C[{\sf K},\O]\,/\,\mathfrak C[{\sf K},\O]_0\,.
\end{equation}
On the other hand, by using exactness of the full crossed product functor \cite[Prop. 3.19]{Wi} and invoking Lemma \ref{stors}, one justifies the isomorphisms
\begin{equation}\label{pasat}
\mathfrak A[{\sf K},\O]\,/\,\mathfrak A[{\sf K},\O]_0\cong\big(\A[{\sf K},\O]\,/\,\A[{\sf K},\O]_0\big)\!\rtimes\!\G\cong\mathcal C(\O_\infty)\!\rtimes_\th\!\G\,.
\end{equation}
We did not indicate each time the actions defining the various crossed products, because they are natural; the one involved in the last term in \eqref{pasat}
is induced by $\vartheta$ on the corona space $\O_\infty$ of $\O$ as described in Section \ref{firica}\,.

By the isomorphisms \eqref{lasat} and \eqref{pasat}, one has to investigate the invertibility issue in the crossed product $\mathcal C(\O_\infty)\!\rtimes_\th\!\G$\,, and this is the core of the proof. We follow closely the abstract approach of \cite{NP,Ro}. 

\begin{Definition}\label{exhausting}
Let $\CC$ be a $C^*$-algebra. A family $\,\mathcal G:=\big\{\Pi_i:\CC\rightarrow\CC_i\mid i\in I\big\}$ of morphisms of $C^*$-algebras is 
\begin{itemize}
\item
{\rm faithful} if $\,\bigcap_{i\in I}\ker(\Pi_i)=\{0\}$ (this is equivalent to $\,\p\!\Phi\!\p_{\CC}\,=\,\sup_{i\in I}\!\p\!\Pi_{i}(\Phi)\!\p_{\CC_{i}}\,$ for every $\Phi\in\mathfrak C$)\,.
\item
{\rm strictly norming} if for any $\Phi\in\CC$ there exists $k\in I$ such that  
$\,\p\!\Phi\!\p_{\CC}\,=\,\p\!\Pi_{k}(\Phi)\!\p_{\CC_{k}}$\,.
\item
{\rm exhaustive} if every primitive ideal $\mathfrak P$ of $\CC$ contains at least one of the ideals $\,\ker(\Pi_i)$\,.
\end{itemize}
\end{Definition}

It has been shown in \cite{NP} that 
\begin{equation*}\label{neata}
{\rm exhaustive}\,\Longrightarrow\,{\rm strictly\ norming}\,\Longrightarrow\,{\rm faithful}
\end{equation*}
and none of the implications is an equivalence. However, if $\CC$ is separable, "strictly norming" and "exhaustive" are equivalent. 
The role of the strictly norming property is put into evidence by the next result, taken from  \cite[Sect. 3]{NP} and \cite{Ro} (cf. also \cite{Ex}):

\begin{Proposition}\label{macmac}
The family $\,\mathcal G$ is strictly norming if and only if the invertibility of an arbitrary element $\Phi$ of $\,\CC$ is equivalent to the invertibility of each of the elements $\Pi_i(\Phi)\in \CC_i$ and also equivalent to
\begin{equation}\label{vetta}
{\rm sp}(\Phi)=\bigcup_{i\in I}{\rm sp}\big[\Pi_i(\Phi)\big]\,,\quad\forall\,\Phi\in \CC\,.
\end{equation}
\end{Proposition}

Inverses and the spectra above are computed in the minimal unitalizations of $\CC$ and $\CC_i$\,, respectively (or, equivalently, in any larger unital $C^*$-algebra). Passing from the $C^*$-algebras to their minimal unitalizations is simple; it is described in \cite[Sect. 3]{NP}.

A similar characterization of faithfulness is also given in \cite{NP}, where the extra condition $\,\sup_{i}\!\p\!\Pi_i(\Phi)^{-1}\!\p_{\CC_i}<\infty\,$ should be added to the element-wise invertibility and in the right-hand side of \eqref{vetta} a closure is needed. So one gets better results for strictly norming than for faithful families. Although exhaustiveness is not involved directly in the spectral results, it often looks easier to check that the strictly norming property.

\begin{Proposition}\label{aceastta}
Let $\big(\O_\infty,\vartheta,\G\big)$ be a compact quasi-regular dynamical system and $\{\o_i\mid i\in I\}$ a sufficient family of points of $\,\O_\infty$\,. For every $\,i\in I$ consider the restriction morphism
\begin{equation*}\label{strik}
\pi_i:\mathcal C(\O_\infty)\to \mathcal C({\sf Q}^{\o_i})\,,\quad \pi_i(b):=b|_{{\sf Q}^{\o_i}}
\end{equation*} 
and then (cf. Remark \ref{mamsatu})
\begin{equation*}\label{crux}
\Pi_i:=\pi_i^\rtimes:\mathcal C(\O_\infty)\!\rtimes\!\G\to \mathcal C({\sf Q}^{\o_i})\!\rtimes\!\G\,.
\end{equation*}
Then $\,\mathcal G:=\big\{\Pi_i\mid i\in I\big\}$ is an exhausting family of morphisms.
\end{Proposition}

\begin{proof}
First we identify the kernel of $\Pi_i$\,. Obviously one has $\,\ker(\pi_i)=\mathcal C^{{\sf Q}^{\o_i}}\!(\O_\infty)$\,. Since the crossed product functor is exact, one gets $\,\ker(\Pi_i)=\mathcal C^{{\sf Q}^{\o_i}}\!(\O_\infty)\!\rtimes\!\G$\,. So, to show that our family is exhaustive, one has to check that every primitive ideal $\mathfrak P$ of $\mathcal C(\O_\infty)\!\rtimes\!\G$ contains $\mathcal C^{{\sf Q}^{\o_k}}\!(\O_\infty)\!\rtimes\!\G$ for some $k\in I$\,.

Let $P:=p\!\rtimes\!U$ be an irreducible representation of $\mathcal C(\O_\infty)\!\rtimes\!\G$ with $\ker(P)=\mathfrak P$\,.
Since the dynamical system is quasi-regular, the representation $P$ lives on a quasi-orbit ${\sf Q}^\o$\,, meaning that $\ker(p)=\mathcal C^{{\sf Q}^{\o}}\!(\O_\infty)$\,. It follows from this and from formula \eqref{maximilian} that $\,\mathfrak P\,\supset\,\mathcal C^{{\sf Q}^{\o}}\!(\O_\infty)\!\rtimes\!\G$\,. 

But $\big\{{\sf Q}^{\o_i}\mid i\in I\big\}$ is a covering of $\O_\infty$\,, so $\o\in{\sf Q}^{\o_k}$ for some index $k$\,, which implies ${\sf Q}^{\o}\subset{\sf Q}^{\o_k}$ and then $\,\mathcal C^{{\sf Q}^{\o}}\!(\O_\infty)\supset\mathcal C^{{\sf Q}^{\o_k}}\!(\O_\infty)$\,. Therefore
\begin{equation*}\label{paci}
\mathfrak P\,\supset\,\mathcal C^{{\sf Q}^{\o}}\!(\O_\infty)\!\rtimes\!\G\supset\mathcal C^{{\sf Q}^{\o_k}}\!(\O_\infty)\!\rtimes\!\G=\ker(\Pi_k)\,.
\end{equation*}
\end{proof}

To make the connection with the operators $H^{\o_i}$ and to finish the proof of our Theorem \ref{labba}, we are only left with checking that for any point $\o\in\O_\infty$
$$
{\sf Sch^\o}:\mathcal C\big({\sf Q}^\o\big)\!\rtimes\!\G\to\mathbb B\!\[L^2(\G)\]\,,\quad{\sf Sch^\o}(\Psi):={\sf Sch}(\Psi^\o)
$$
is a faithful representation. It is obviously the composition of the faithful representation 
$$
{\sf Sch}:\mathcal{LUC}(\G)\!\rtimes_{\sf l}\G\to\mathbb B\!\[L^2(\G)\]
$$
with the morphism 
$$
\Gamma^\o:=\big(\gamma^\o\big)^{\!\rtimes}:\mathcal C\big({\sf Q}^\o\big)\!\rtimes\!\G\to\mathcal{LUC}(\G)\!\rtimes_{\sf l}\G\,,
$$ 
where at the Abelian level
\begin{equation*}\label{belian}
\big[\gamma^\o(b)\big](q):=b\big[\vartheta_q(\o)\big]\,.
\end{equation*}
Since $\gamma^\o$ is injective, $\Gamma^\o$ is injective too.

\section{On quasi-regularity}\label{anulatus}

Theorem \ref{labba} relies on the implicit requirement that the dynamical system $\big(\O_\infty,\vartheta,\G\big)$ be quasi-regular. There is no counter-example to this property, so it might simply be proved to hold always and then disappear from the assumptions of Theorem \ref{labba} and also from many other results in the theory of crossed products where it plays a role. Experts believe that a dynamical system failing to be quasi-regular, if it exists, should be very pathological.
We cite from \cite[pag. 184]{Wi}: "Most systems we are interested in will be quasi-regular and it may even be the case that all are." This even refers to actions of locally compact groups on {\it non-abelian} $C^*$-algebras. 

In \cite{Wi} quasi-regularity is studied in connection with the quasi-orbit structure of the dynamical system. Let $\mathbf Q(\O_\infty)$  be the quasi-orbit space of $\big(\O_\infty,\vartheta,\G\big)$ with the quotient topology. It is constructed from the equivalence relation $\o\sim\o'$ if and only if ${\sf Q}^\o={\sf Q}^{\o'}$. Equivalently, it is the $T_0$-ization \cite[Def. 6.9]{Wi} of the orbit space $\mathbf O(\O_\infty):=\O_\infty/\G$ (this means that in the topological space $\mathbf O(\O_\infty)$\,, which might not be $T_0$\,, points are identified if they have the same closure).  By \cite[Prop. 6.21]{Wi}, $\big(\O_\infty,\vartheta,\G\big)$ {\it is quasi-regular if either $\mathbf Q(\O_\infty)$ is second countable (which follows if $\O_\infty$ is second countable) or if $\mathbf Q(\O_\infty)$ is almost Hausdorff (see bellow)}. Anyhow, if the initial left invariant $C^*$-algebra $\A[{\sf K},\O]$ is separable, quasi-regularity holds. 

Second countability excludes many interesting examples, so let us spell out the second sufficient condition. A topological space is called {\it almost Hausdorff} if 
every nonempty closed subspace has a relatively open nonempty Hausdorff subspace. Other characterizations of this property may be found in \cite[Sect. 6.1]{Wi}.
It seems, however, difficult to check this property in our case without a good understanding of the quasi-orbit structure of $\big(\O_\infty,\vartheta,\G\big)$\,. Anyhow, we have regularity if the quasi-orbit space is Hausdorff; we will encounter such a situation in Section \ref{anulitus}.

\begin{Remark}\label{separam}
{\rm Results from \cite{LSe} corroborated with Proposition \ref{macmac} seems to indicate that $\big(\O^{luc}_\infty,\vartheta,\mathbb Z^n\big)$\,, corresponding to the crossed product $\ell^\infty(\Z^n)\!\rtimes\!\Z^n$ and to standard band-dominated operators, is quasi-regular. It would be nice to prove such a result for an arbitray locally compact (maybe non-commutative group) $\G$\,.}
\end{Remark}

\section{Some examples}\label{anulitus}

We put now into evidence situations in which the quasi-orbit space is understood; this solves quasi-regularity and allows a concrete application of Theorem \ref{labba}.

Let us assume that ${\sf K}$ is a compact {\it normal} subgroup of $\G$\,. Then ${\sf H}:=\G/{\sf K}$ is a locally compact group, the $C^*$-algebras $\mathcal{LUC}({\sf H})$ and $\mathcal{RUC}({\sf H})$ are available, as well as left and right translations with elements $y{\sf K}={\sf K}y$ of ${\sf H}$\,. We set
\begin{equation*}\label{chaplin}
\mathcal{SO}({\sf H}):=\{b\in\mathcal{LUC}({\sf H})\mid {\sf r}_{y{\sf K}}(b)-b\in\mathcal C_0({\sf H})\,,\ \forall\,y{\sf K}\in{\sf H}\}
\end{equation*}
and call its elements {\it slowly oscillation functions on ${\sf H}$\,.} There is a similar class defined using left translations, but we are not going to need it.
Since $\mathcal C_0({\sf H})$ is an invariant ideal and left translations commute with right translations, it is easy to check that $\mathcal{SO}({\sf H})$ is a left-invariant $C^*$-subalgebra of $\mathcal{LUC}({\sf H})$\,. It is unital and contains $\mathcal C_0({\sf H})$\,. Therefore its Gelfand spectrum $\O^\mathcal{SO}$ is an ${\sf H}$-compactification of ${\sf H}=\G/{\sf K}$\,. Using previous notation one has $\mathcal{SO}({\sf H})=\mathcal C\left\<\O^\mathcal{SO}\right\>\cong\mathcal C\big(\O^\mathcal{SO}\big)$\,.

\begin{Lemma}\label{scar}
One has a right invariance data $\Upsilon:=\Big({\sf K},\O^\mathcal{SO},\vartheta\Big)$\,, where every $\vartheta_x$ is trivial on $\O^\mathcal{SO}_\infty:=\O^\mathcal{SO}\setminus{\sf H}$\,. An element $a$ of $\mathcal{LUC}(\G)$ belongs to the associated $C^*$-algebra $\A\big[{\sf K},\O^\mathcal{SO}\big]$ if and only if
\begin{equation*}\label{chaplan}
a(xz)=a(x)\,,\quad\forall\,x\in\G\,,\ z\in{\sf K}
\end{equation*}
and for every $\epsilon>0$ and $z\in\G$ there exists a compact subset $\kappa$ of $\,\G/{\sf K}$ such that
\begin{equation*}\label{claftan}
|a(xz)-a(x)|\le\epsilon\quad{\rm if}\,\ x{\sf K}\notin\kappa\,.
\end{equation*}
\end{Lemma}

\begin{proof}
For the first statement one has to show that the (Higson-type) corona space $\,\O^\mathcal{SO}_\infty:=\O^\mathcal{SO}\setminus{\sf H}\,$ is only formed of fixed points under the action of the group $\,{\sf H}$\,. This is a particular case of \cite[Prop. 3.30]{Man}.

The description of the associated $C^*$-algebra follows straightforwardly from the definitions.
\end{proof}

The next result can be easily deduced from these preparations and from Theorem \ref{labba}. It describes the essential spectrum and the Fredholm property of a Schr\"odinger (or pseudo-differential) operator "with slowly oscillating coefficients" in terms of a family of convolution operators. By Lemma \ref{scar}, the orbits in the corona space are singletons, so the (quasi-)orbit space of $\O^\mathcal{SO}_\infty$ can be identified to $\O^\mathcal{SO}_\infty$ itself. It is Hausdorff, thus almost Hausdorff, so the dynamical system is qusi-regular.

\begin{Corollary}\label{rolar}
Let ${\sf K}$ be a compact normal subgroup of the amenable locally compact group $\G$\,. For every $\Phi\in L^1\Big(\G;\A\big[{\sf K},\O^\mathcal{SO}\big]\Big)$ and every $\o\in\O^\mathcal{SO}_\infty$\,, using notations from Section \ref{flokoshin}, let $\Phi^\o(\cdot):=\widetilde\Phi_{\sf K}(\o;\cdot)\in L^1(\G)$\,. Set $H:={\sf Sch}(\Phi)$ and $H^\o:={\sf Conv}(\Phi^\o)$\,.
\begin{enumerate}
\item
One has 
$$
{\rm sp}_{\rm ess}(H)=\!\bigcup_{\o\in\O^\mathcal{SO}_\infty}\!{\sp}(H^\o)\,.
$$
\item
The operator $H$ is Fredholm if and only if all the operators $H^\o$ are invertible.
\end{enumerate}
\end{Corollary}

To get very explicit results, let us suppose that ${\sf K}=\{\e\}$ and $\Phi=a\otimes\varphi$\,, with $a\in\A\big[\{\e\},\O^\mathcal{SO}\big]=\mathcal{SO}(\G)$ and $\varphi\in L^1(\G)$\,. Then $\Phi^\o=\widetilde a(\o){\sf Conv}(\varphi)$\,, where $\widetilde a$ is the extension to $\O^\mathcal{SO}$ of $a:{\sf G}\to\mathbb C$\,. Obviously $\,{\sp}(H^\o)=\widetilde a(\o){\sp}[{\sf Conv}(\varphi)]$\,, so one gets 
$$
{\rm sp}_{\rm ess}(H)=\Big[\!\bigcup_{\o\in\O^\mathcal{SO}_\infty}\!\!\widetilde a(\o)\Big]{\rm sp}[{\sf Conv}(\varphi)]\,.
$$ 
On the other hand, if the group $\G$ is connected, it is true that 
$$
\bigcup_{\o\in\O^\mathcal{SO}_\infty}\!\!\widetilde a(\o)=\big[\liminf_{x\to\infty} a(x),\limsup_{x\to\infty}a(x)\big]\,,
$$
so {\it one gets the essential spectrum of $H$ by scaling the spectrum of the convolution operator $\,{\sf Conv}(\varphi)$ by the asymptotic range of the slowly oscillating function $a$}\,.

The operator $H^\o$ is invertible if and only if ${\sf Conv}(\varphi)$ is invertible and $\,\widetilde a(\o)$ is non-null. So finally {\it $H$ is Fredholm if and only if $\,{\sf Conv}(\varphi)$ is invertible and $0$ is not contained in the asymptotic range of $a$\,}.

One can treat similarly the case $\Phi(q;x)=\varphi(x)+a(q)$\,. For example
$$
{\rm sp}_{\rm ess}\big[{\sf Sch}(\Phi)\big]={\rm sp}[{\sf Conv}(\varphi)]+\big[\liminf_{x\to\infty} a(x),\limsup_{x\to\infty}a(x)\big]\,.
$$

\begin{Remark}\label{frustuc}
{\rm If $\G$ is Abelian, than ${\sf Conv}(\varphi)$ is unitarily equivalent, via the Fourier transformation ${\fscr F}$, to the operator of multiplication with $\,\widehat{\varphi}:={\fscr F}(\varphi)$ in the $L^2$ Hilbert space of the dual group $\wG$\,. The spectrum of this operator is the closure of the range $\widehat{\varphi}\big(\wG\big)$\,; it is invertible if and only if this closed set does not contain the point $\{0\}$\,. This makes the results above even more explicit.
}
\end{Remark}

\begin{Remark}\label{ducts}
{\rm Let us suppose that $\G=\G_1\times\G_2$ is the product of two locally compact groups. In general $\mathcal{SO}({\sf G})\ne\mathcal{SO}({\sf G_1})\otimes\mathcal{SO}({\sf G_2})$\,. The tensor product $\mathcal{SO}({\sf G_1})\otimes\mathcal{SO}({\sf G_2})$ is still very manageable. Its Gelfand spectrum can be identified to the topological product $\O^{\mathcal{SO},1}\!\times\O^{\mathcal{SO},2}$ (obvious notations). Since the natural action of $\G_1\times\G_2$ is a product action, it is easy to understand the orbit and the quasi-orbit structure in the corona space 
$$
\Big(\O^{\mathcal{SO},1}\!\times\O^{\mathcal{SO},2}\Big)\setminus\big(\G_1\times\G_2\big)=\big(\O^{\mathcal{SO},1}_\infty\!\times\G_2\big)\sqcup\big(\G_1\times\O^{\mathcal{SO},2}\big)\sqcup\Big(\O^{\mathcal{SO},1}_\infty\!\times\O^{\mathcal{SO},2}_\infty\Big)\,.
$$
The orbits of the last component are just singletons $\{(\o_1,\o_2)\}$\,, already closed. Those of $\O^{\mathcal{SO},1}_\infty\!\times\G_2$ have the form $\{\o_1\}\times\G_2$ for some $\o_1\in\O^{\mathcal{SO},1}_\infty$, with closure $\{\o_1\}\times\O^{\mathcal{SO},2}$ and those of $\O^{\mathcal{SO},1}_\infty\!\times\G_2$ have the form $\G_1\times\{\o_2\}$ for some $\o_2\in\O^{\mathcal{SO},2}_\infty$, with closure $\O^{\mathcal{SO},1}\times\{\o_2\}$\,. The singleton quasi-orbits are covered by other quasi-orbits, so they can be neglected. A sufficient family of points in the corona space is 
$$
\Big\{\o_1\times\e_2\,\big\vert\,\o_1\in\O^{\mathcal{SO},1}\Big\}\cup\Big\{\e_1\times\o_2\,\big\vert\,\o_2\in\O^{\mathcal{SO},2}\Big\}\,,
$$
where $\e_j$ denotes the unit of the group $\G_j$\,.  We leave to the motivated reader the work of writing down the statements of Theorem \ref{labba} in this particular case.
}
\end{Remark}

Finally we indicate a more complex example for which the quasi-orbit structure can be computed. Then the spectral results can be easily deduced from Theorem \ref{labba}; some particular cases can already be found in \cite{Ma}. The critical result allowing to understand the quasiorbits in the non-commutative case is proven (in greater generality) in \cite{Man}, to which we send for further details and references. 

We consider only the case ${\sf K}=\{\e\}$\,, for simplicity.

A left-invariant $C^*$-subalgebra $\mathcal B$ of $\mathcal{LUC}(\G)$ will be called {\it $\G$-simple} if it contains no proper left-invariant ideal; this is equivalent with its Gelfand spectrum $\O^{\mathcal B}$ being a minimal dynamical sistem (all orbits are dense). The elements of such a $\G$-simple $C^*$-algebra are called {\it minimal functions}. Large $\G$-simple $C^*$-algebras exist. All the almost periodic functions form such an algebra (the Gelfand spectrum is the Bohr "compactification" of $\G$), but there are other explicit examples involving distal or almost automorphic functions.

Let us pick a $\G$-simple $C^*$-algebra $\mathcal B$ and denote by $\big<\mathcal B\cdot\mathcal{SO}(\G)\big>$ the smallest $C^*$-algebra containing both $\mathcal B$ and $\mathcal{SO}(\G)$ (it is generated by products $bc$ with $b\in\mathcal B$ and $c\in\mathcal{SO}(\G)$)\,. It is easy to show thet it is a left-invariant subalgebra of $\mathcal{LUC}(\G)$\,. It is less easy, but still true \cite{Man}, that 
$$
\O^{\big<\mathcal B\cdot\mathcal{SO}(\G)\big>}=\G\sqcup\Big(\O^\mathcal B\times\O^\mathcal{SO}_\infty\Big)
$$
and that the (already closed) orbits of $\O^{\big<\mathcal B\cdot\mathcal{SO}(\G)\big>}_{\,\infty}$ are of the form $\O^\mathcal B\times\{\o'\}$ for some point $\o'\in\O^\mathcal{SO}_\infty$ of the Higson corona. Therefore the orbit and the quasi-orbit spaces coincide and can be identified to the Hausdorff compact space $\O^\mathcal{SO}_\infty$\,; quasi-regularity is insured in a simple way. Choosing some $\o_0\in\O^\mathcal B$, the family $\big\{(\o_0,\o')\mid \o'\in\O^\mathcal{SO}_\infty\big\}$ is sufficient. If $a\in\big<\mathcal B\cdot\mathcal{SO}(\G)\big>$\,, say $a=bc$ with $b\in\mathcal B$ and $c\in\mathcal{SO}(\G)$\,, then $a^{(\o_0,\o')}(x)=c(\o')\widetilde b\big[\vartheta'_x(\o_0)\big]$\,, where $\vartheta'$ is the (minimal) action of $\G$ on $\O^\mathcal B$. Thus the $x$-dependence of this function comes from the minimal component $b$\,; the slowly oscillating part $c$\,, evaluated in $\o'$\,, serves as a coupling constant.


\medskip

Marius M\u antoiu:
  \endgraf
  Departamento de Matem\'aticas, Universidad de Chile,
  \endgraf
  Casilla 653, Las Palmeras 3425, Nunoa,  Santiago, Chile.
    \endgraf
  {\it E-mail address:} {\rm mantoiu@uchile.cl}

\end{document}